\documentclass[12pt, reqno]{amsart}
\usepackage{mathrsfs}
\usepackage{geometry}
\usepackage{titletoc}
\usepackage{mathpazo}
\usepackage{amsmath}
\usepackage{amssymb}
\usepackage{enumitem}
\usepackage{mathtools}
\usepackage[table]{xcolor}
\usepackage[all]{xy}
\usepackage{tikz}
\usepackage{tikz-cd}
\usepackage{indentfirst}
\usepackage{setspace}
\usepackage[colorlinks,linkcolor=red,anchorcolor=green,citecolor=blue]{hyperref}
\hypersetup{linktocpage = true}
\usepackage{rotating}

\usepackage{ytableau}
\usepackage{longtable}
\newcolumntype{M}[1]{>{\centering\arraybackslash}m{#1}}

\allowdisplaybreaks[4]

\usepackage{mathrsfs}
\DeclareFontFamily{OMS}{rsfs}{\skewchar\font'60}
\DeclareFontShape{OMS}{rsfs}{m}{n}{<-5>rsfs5 <5-7>rsfs7 <7->rsfs10 }{}
\DeclareSymbolFont{rsfs}{OMS}{rsfs}{m}{n}
\DeclareSymbolFontAlphabet{\scr}{rsfs}
\DeclareSymbolFontAlphabet{\scr}{rsfs}
\usepackage[T1]{fontenc}
\theoremstyle{plain}
\newtheorem{thm}{Theorem}[section]
\newtheorem{lemma}[thm]{Lemma}
\newtheorem{prop}[thm]{Proposition}
\newtheorem{cor}[thm]{Corollary}

\theoremstyle{definition}

\newtheorem{remark}[thm]{Remark}

\renewcommand{\bar}{\overline}
\newcommand{\eps}{\varepsilon}
\renewcommand{\phi}{\varphi}

\newcommand{\bd}{\begin{enumerate}}
\newcommand{\ed}{\end{enumerate}}

\renewcommand{\tilde}{\widetilde}

\newcommand{\CC}{\mathbb {C}}

\numberwithin{equation}{section}

\usepackage{fancyhdr}
\title{A Demailly-type approximation of singular Finsler metrics}
\author{Zhuo Liu}
\address{Mathematical Science Research Center, Chongqing University of Technology, No. 69, Hongguang Avenue, Banan District, Chongqing 400054, China.}
\email{liuzhuo@cqut.edu.cn; liuzhuo@amss.ac.cn}

\begin{document}

\begin{abstract}
We establish a Demailly-type approximation theorem for singular Finsler metrics satisfying the multiple coarse \(L^p\) extension property on holomorphic vector bundles. As an application, we obtain a new proof of the fact that the multiple coarse \(L^p\) extension property characterizes Griffiths positivity, as shown by Deng--Wang--Zhang--Zhou.
\end{abstract}

\maketitle

\section{Introduction}

Plurisubharmonicity and Griffiths/Nakano positivity play fundamental roles in several complex variables and complex geometry, yielding numerous important results, including H\"ormander's \(L^2\) estimates for the \(\overline{\partial}\)-equations \cite{Dem-82, Hor65}, the Ohsawa--Takegoshi \(L^2\) extension theorem \cite{OT87}, and the optimal \(L^2\) extension theorem \cite{Blocki13, GuanZhou12, GuanZhou15}. Plurisubharmonic functions, which need not be smooth, offer significant advantages in addressing many problems. Similarly, there has been growing interest in singular metrics with positivity properties on vector bundles (\cite{BP08, Cataldo98, PT18, Rau15}). 

Let \(E\to X\) be a holomorphic vector bundle of rank \(r\) over a complex manifold \(X\), and let \(h\) be a singular Finsler metric on \(E\), i.e., a function \(|\cdot|_h:E\to[0,+\infty]\), homogeneous of degree one on each fiber, finite almost everywhere and positive on non-zero vectors almost everywhere. The dual Finsler metric \(h^*\) on \(E^*\) is defined by
\[
|u|_{h^*}^2=\sup_{v\in E_x\setminus\{0\}}\frac{|\langle u,v\rangle|^2}{|v|_h^2},\qquad u\in E_x^*.
\]
A singular Finsler metric \(h\) on \(E\) is said to be \emph{Griffiths semi-negative} if for every local holomorphic section \(s\) of \(E\), the function \(|s|^2_h\) is plurisubharmonic. It is said to be \emph{Griffiths semi-positive} if \((E^*,h^*)\) is Griffiths semi-negative.

Recently, in \cite{DNW21, DNWZ22, DWZZ18}, Deng, Ning, Wang, Zhang, Zhou established the converse $L^2$ theory by giving alternative characterizations of plurisubharmonicity and Griffiths/Nakano positivity in terms of various $L^2$-conditions for $\overline{\partial}$. These characterizations are significant precisely because they avoid metric differentiations, thereby offering workable definitions of various positivity notions for singular Hermitian metrics on holomorphic vector bundles.  

A key problem in the study of singular metrics on holomorphic vector bundles is the approximation of singular metrics with positivity properties. Previous results in this direction can be found in \cite{BP08, DNWZ23, Liu26, Rau15}, where the approximating metrics are constructed locally by convolution. In contrast, Demailly's approximation theorem \cite{Dem92}, which states that every plurisubharmonic function can be approximated by the Bergman kernels of weighted \(L^2\) spaces, is a fundamental tool in complex analytic geometry. In this paper, we study a vector-bundle analogue of it for singular Finsler metrics satisfying the \emph{multiple coarse \(L^p\) extension property} introduced by Deng--Wang--Zhang--Zhou \cite[Definition 1.1]{DWZZ18}.

\begin{thm}\label{thm:abstract}
Let \(p>0\), and let \(E\to X\) be a holomorphic vector bundle endowed with a  singular Finsler metric \(h\) such that  \(|u|_{h^*}\) is upper semi-continuous for every local holomorphic section \(u\) of \(E^*\).  Denote by \(H^p(X,E,h)\) the space of holomorphic sections \(s\) of \(E\) with \(\|s\|_{L^p(h)}<+\infty\). Assume that \((E,h)\) satisfies the multiple coarse \(L^p\) extension property, i.e., for each \(m\ge1\), every \(x\in X\), and every \(a\in E_x\setminus\{0\}\), there is \(s\in H^p(X,E^{\otimes m},h^{\otimes m})\) with \(s(x)=a^{\otimes m}\) and
\[
\|s\|_{L^p(h^{\otimes m})}^p\le C_m|a|_{h(x)}^{pm},
\qquad \lim_{m\to\infty}\frac{\log C_m}{m}=0.
\] For each \(m\ge1\), define \(h_m^*\) on \(E^*\) by
\begin{equation}\label{eq:h_m-tensor}
|u|_{h_m^*}:=\left(\left(\sup_{\|s\|_{L^p(h^{\otimes m})}\le1}|\langle u^{\otimes m},s(x)\rangle|\right)^{*}\right)^{1/m},\qquad u\in E_x^*,
\end{equation}
where \((\cdot)^*\) denotes the upper semi-continuous regularization. Then \(h_m^*\) is a well-defined singular Finsler metric on \(E^*\), Griffiths semi-negative, and satisfies
\[
C_m^{-1/(pm)}|u|_{h^*}\le |u|_{h_m^*}\le C_r^{1/(pm)}\sup_{B(x,r)}|u|_{h^*},
\]
with \(C_r=1/\operatorname{Vol}(B(x,r))\), for every local holomorphic section \(u\) of \(E^*\), every \(x\in X\), and sufficiently small \(r>0\). Moreover, \(|u|_{h_m^*}\to|u|_{h^*}\) pointwise and in \(L^1_{\mathrm{loc}}\) as \(m\to\infty\).
\end{thm}

\begin{remark}\label{rem:sm-version}
\begin{enumerate}
\item[(i)] If \(p=2\) and \(h\) is a Hermitian metric, then \(H^2(X,E,h)\) is a Hilbert space. Let \(\{\sigma_j\}_{j\ge1}\) be an orthonormal basis. Then the expression \eqref{eq:h_m-tensor} for \(h_m^*\) reduces to the explicit form 
\[
|u|^2_{h_m^*} =
\left(
\sum_{j\ge1}
|\langle u^{\otimes m},\sigma_j^{(m)}(x)\rangle|^2
\right)^{1/m},
\qquad u\in E_x^*.
\]
\item[(ii)] When \(p=2\) and \(E=L\) is a line bundle with \(h=e^{-\varphi}\), the Ohsawa--Takegoshi \(L^2\) extension theorem for the plurisubharmonic weight \(m\varphi\) directly yields the multiple coarse\(L^2\) extension property, and Theorem~\ref{thm:abstract} reduces to the classical Demailly approximation theorem for plurisubharmonic functions.
\end{enumerate}
\end{remark}

As an immediate consequence, we recover the following result of Deng--Wang--Zhang--Zhou, with a new proof based on the approximation theorem above.

\begin{cor}[{\cite[Theorem 1.2]{DWZZ18}}]\label{cor:DWZZ}
Let \(p\ge1\), and let \(E\to X\) be a holomorphic vector bundle of rank \(r\), endowed with a singular Finsler metric \(h\) such that \(|u|_{h^*}\) is upper semi-continuous for every local holomorphic section \(u\) of \(E^*\). Assume that \((E,h)\) satisfies the multiple coarse \(L^p\) extension property. Then \(h\) is Griffiths semi-positive.
\end{cor}

One of the fundamental motivations of Demailly's approximation theorem is to construct plurisubharmonic approximations with analytic singularities that preserve the multiplier ideal sheaf. This addresses the existence problem of analytic weights for multiplier ideal sheaves and naturally leads to Demailly's strong openness conjecture \cite{Demailly2000, DK2001, Kiselman}. In \cite{GuanZhouSOC2013}, Guan--Zhou affirmatively solved Demailly's strong openness conjecture and  established the existence of analytic weights for multiplier ideal sheaves. Thus we investigate the analogous problems in the vector bundle setting. More precisely, we replace multiplier ideal sheaves by multiplier submodule sheaves \(\mathcal E(h)\) associated with singular Finsler metrics \(h\), where \(\mathcal E(h)\) is defined by
\[
\mathcal E(h)_x:=\{u\in\mathcal O(E)_x:\ |u|_h^2\text{ is integrable in some neighborhood of }x\},
\]
and obtain the following result.

\begin{prop}\label{prop:multiplier-sheaves-strong-openness}
Let \(D\subset\CC^n\) be a bounded domain, \(E\to D\) a trivial holomorphic vector bundle of rank \(r\), and \(h\) a Hermitian metric on \(E\) such that \(|u|_{h^*}\) is upper semi-continuous for every local holomorphic section \(u\) of \(E^*\). Assume that \((E,h)\) satisfies the multiple coarse \(L^2\) extension property and the strong openness property: for every relatively compact open subset \(U\Subset D\) over which \(E\) is trivial, there exists \(\eps>0\) such that
\[
\mathcal E(h)_x=\mathcal E(h(\det h)^\eps)_x
\]
for all \(x\in U\). Then \(
\mathcal E(h)=\mathcal E\!\left(h_m^{m/(m-1)}\right)
\) for $m$ large enough.
\end{prop}

\begin{remark}
It is well known that multiplier ideal sheaves enjoy several important properties, including coherence \cite{Nadel}, the Ohsawa--Takegoshi $L^2$ extension theorem \cite{OT87}, and the strong openness property \cite{GuanZhouSOC2013, GZ15invent}. For Nakano semi-positive singular Hermitian metrics on holomorphic vector bundles, Inayama \cite{Ina-AG} established coherence, while Liu--Xiao--Yang--Zhou \cite{LXYZ24} proved an \(L^2\) extension theorem and the strong openness property for multiplier submodule sheaves. However, the analogous questions for multiplier submodule sheaves associated with Griffiths semi-positive singular Hermitian metrics remain largely open  (\cite[Conjecture 1.1]{Inayama22} and \cite[Question 1.2]{LX26}). Perhaps Griffiths semi-positivity alone is not sufficient; we suspect that the multiple coarse \(L^2\)-extension property may be sufficient.
\end{remark}

We now consider the special case where \(h\) is a Griffiths semi-positive Hermitian metric on a trivial holomorphic vector bundle \(E\) over a bounded pseudoconvex domain \(D\subset\CC^n\). Repeating the construction of the tensor-power case with \((E^{\otimes m},h^{\otimes m})\) replaced by \((\mathrm{Sym}^mE\otimes\det E,\mathrm{Sym}^mh\otimes\det h)\), we obtain the following Demailly-type approximation.

\begin{thm}\label{thm:second-main-hermitian}
Let \(h\) be a Griffiths semi-positive Hermitian metric on a trivial holomorphic vector bundle \(E\) over a bounded pseudoconvex domain \(D\subset\CC^n\). For each \(m\ge1\), let \(\{\sigma_j^{(m)}\}_{j\ge1}\) be an orthonormal Hilbert basis of
\(
H^2\bigl(D,\mathrm{Sym}^mE\otimes\det E,\ \mathrm{Sym}^mh\otimes\det h\bigr).
\)
Define a Finsler metric \(h_m^*\) on \(E^*\) by
\begin{equation}\label{eq:h_m-sym-hermitian}
|u|_{h_m^*}:=
\left(
\sum_{j\ge1}
\left|\langle u^m,\sigma_j^{(m)}(x)\rangle\right|^2
\right)^{1/(2m)},
\qquad u\in E_x^*.
\end{equation}
Then:
\begin{enumerate}
\item[(i)] \(h_m^*\) is a well-defined singular Finsler metric on \(E^*\), Griffiths semi-negative, and smooth outside an analytic subset of \(D\).
\item[(ii)] For every local holomorphic section \(u\) of \(E^*\), every \(x\in D\), and every sufficiently small \(r>0\),
\[
e^{-C/m}|u|_{h^*}(\det h(x))^{-1/(2m)}
\le
|u|_{h_m^*}
\le
C_r^{1/m}\sup_{B(x,r)}|u|_{h^*}(\det h(x))^{-1/(2m)},
\]
where \(C_r=1/\operatorname{Vol}(B(x,r))\) with respect to a fixed Lebesgue measure in a local coordinate chart.
\item[(iii)] \(|u|_{h_m^*}\) converges to \(|u|_{h^*}\) pointwise and in \(L^1_{\mathrm{loc}}\) as \(m\to\infty\).
\end{enumerate}
\end{thm}

\section{Proof of main results} 
\subsection{Proof of Theorem \ref{thm:abstract}} 
 First, we show that the supremum in \eqref{eq:h_m-tensor} is finite and that \(h_m^*\) is Griffiths semi-negative.
Fix a local holomorphic section \(u\) of \(E^*\), a point \(x\in X\), and \(r>0\) such that \(B(x,r)\Subset X\). For any \(s\in H^p(X,E^{\otimes m},h^{\otimes m})\) with \(\|s\|_{L^p(h^{\otimes m})}\le1\), the function \(|\langle u^{\otimes m},s\rangle|^p\) is plurisubharmonic. By the mean value inequality and the duality inequality \(|\langle u^{\otimes m},s\rangle|\le|u|_{h^*}^m|s|_{h^{\otimes m}}\), we obtain
\begin{align*}
|\langle u^{\otimes m},s\rangle|^p(x)
&\le
C_r\int_{B(x,r)}|\langle u^{\otimes m},s\rangle|^p\,dV \\
&\le
C_r\sup_{B(x,r)}|u|_{h^*}^{pm}\int_{B(x,r)}|s|^p_{h^{\otimes m}}
\le
C_r\sup_{B(x,r)}|u|_{h^*}^{pm},
\end{align*}
where \(C_r=\frac{1}{\operatorname{Vol}(B(x,r))}\), and the last inequality uses \(\|s\|_{L^p(h^{\otimes m})}\le1\). Since \(|u|_{h^*}\) is upper semi-continuous, the supremum in \eqref{eq:h_m-tensor} is finite and the upper semi-continuous regularization is well-defined. In addition, we have
\[
|u|_{h_m^*}\le C_r^{1/(pm)}\sup_{B(x,r)}|u|_{h^*}.
\]
Moreover, since \(|\langle u^{\otimes m},s\rangle|^{1/m}\) is plurisubharmonic, we obtain that \(|u|_{h_m^*}\) is plurisubharmonic, and \(h_m^*\) is Griffiths semi-negative.

It remains to prove the lower bound. Fix \(x\in X\) and \(u\in E_x^*\setminus\{0\}\). By the definition of the dual Finsler norm, for every \(\eps>0\) there exists \(a_\eps\in E_x\setminus\{0\}\) such that
\[
\langle u,a_\eps\rangle=1,\qquad |a_\eps|_{h(x)}\le(1+\eps)\,|u|_{h^*(x)}^{-1}.
\]
By the multiple coarse \(L^p\) extension property, there exists
\[
s_\eps\in H^p(X,E^{\otimes m},h^{\otimes m})
\]
such that \(s_\eps(x)=a_\eps^{\otimes m}\) and
\[
\|s_\eps\|_{L^p(h^{\otimes m})}^p
\le C_m|a_\eps|_{h(x)}^{pm}
\le C_m(1+\eps)^{pm}|u|_{h^*(x)}^{-pm}.
\]
Then
\[
|u|_{h_m^*}^m
\ge
\frac{|\langle u^{\otimes m},s_\eps(x)\rangle|}{\|s_\eps\|_{L^p(h^{\otimes m})}}
=
\frac{1}{\|s_\eps\|_{L^p(h^{\otimes m})}}
\ge
C_m^{-1/p}(1+\eps)^{-m}|u|_{h^*(x)}^{m}.
\]
Letting \(\eps\to0\), we obtain
\[
|u|_{h_m^*}\ge C_m^{-1/(pm)}|u|_{h^*}.
\]
Since \(\lim_{m\to\infty} C_m^{-1/m}=1\) and \(|u|_{h^*}\) is upper semi-continuous, we obtain that \(|u|_{h_m^*}\) converges to \(|u|_{h^*}\) pointwise and in \(L^1_{\mathrm{loc}}\) topology as \(m\to\infty\). We complete the proof.

\begin{remark}
If one replaces \(E^{\otimes m}\) by \(\mathrm{Sym}^mE\) in the multiple coarse \(L^p\) extension property, the same arguments apply with obvious modifications. In fact, the proof does not use the finite rank of \(E\). It also works for Banach vector bundles, provided the tensor powers are completed with respect to a suitable tensor norm (for instance, the projective tensor norm) so that the duality inequality
\[
|\langle u^{\otimes m},s\rangle|\le |u|_{h^*}^m\,|s|_{h^{\otimes m}}
\]
remains valid.
\end{remark}

\subsection{Proof of Proposition \ref{prop:multiplier-sheaves-strong-openness}}

Let \(E\to X\) be a holomorphic vector bundle of rank \(r\) over a complex manifold \(X\), endowed with a singular Finsler metric \(h\) such that \(|u|_{h^*}\) is upper semi-continuous for every local holomorphic section \(u\) of \(E^*\).   Assume that \((E,h)\) satisfies the multiple coarse\(L^2\) extensionproperty.  By Remark~\ref{rem:sm-version}-(i), the Finsler metric \(h_m^*\) of Theorem~\ref{thm:abstract} admits the following equivalent explicit form: for every \(u\in E_x^*\),
\begin{equation}\label{eq:h_m-infinite-sum}
|u|^2_{h_m^*}=
\left(
\sum_{j\ge1}
|\langle u^{\otimes m},\sigma_j^{(m)}(x)\rangle|^2
\right)^{1/m},
\end{equation}
where \(\{\sigma_j^{(m)}\}_{j\ge1}\) is an orthonormal Hilbert basis of \(H^2(X,E^{\otimes m},h^{\otimes m})\). In particular, \(h_m^*\) is Griffiths semi-negative, smooth outside an analytic subset of \(X\), and satisfies \[
C_m^{-1/m}|u|^2_{h^*}\le |u|^2_{h_m^*}\le C_r^{1/m}\sup_{B(x,r)}|u|^2_{h^*}.
\] 

 \begin{lemma}\label{lem:finite-control}
Let \(E\to X\) be a holomorphic vector bundle over a compact complex manifold \(X\), endowed with a singular Hermitian metric \(h\) such that \(h\)  is locally bounded from below by a positive constant. Let \(\{\sigma_j\}_{j\ge1}\) be any orthonormal Hilbert basis of \(H^2(X,E,h)\). For every relatively compact subset \(K\Subset X\), there exist \(N_K\ge1\) and \(C_K>0\) such that
\[
\sum_{j=1}^\infty |\sigma_j|_h^2
\le
C_K\sum_{j=1}^{N_K} |\sigma_j|_h^2
\qquad\text{on }K.
\]
\end{lemma}

\begin{proof}
Without loss of generality, we may assume that \(K\Subset\CC^n\), \(K=\{\bar z:z\in K\}\) and that \(E\) is trivial on \(K\). 
Let \(\{\sigma_j\}_{j\ge1}\) be an orthonormal Hilbert basis of \(H^2(X,E,h)\). For each \(z\in X\) and each \(1\le j\le r\), let
\[
\mathrm{Eva}^j_z: H^2(X,E,h)\longrightarrow\CC,
\qquad
\mathrm{Eva}^j_z(f):=f^j(z),
\]
be the evaluation map. Since \(h\) has positive lower bound, we obtain that \(\mathrm{Eva}^j_z\) is a bounded linear functional on \( H^2(X,E,h)\) for every \(z\in X\).

By the Riesz representation theorem, there exists an element \(K_j(\cdot,z)\in H^2(X,E,h)\) such that
\[
\mathrm{Eva}^j_z(f)=\langle\!\langle f(\cdot),K_j(\cdot,z)\rangle\!\rangle_h
\]
for every \(f\in H^2(X,E,h)\), where \(\langle\!\langle\cdot,\cdot\rangle\!\rangle_h\) denotes the inner product of \( H^2(X,E,h)\). In particular,
\[
K_i^j(z,w)=\mathrm{Eva}^j_z(K_i(\cdot,w))=\langle\!\langle K_i(\cdot,w),K_j(\cdot,z)\rangle\!\rangle_h,
\]
and
\[
\sigma_k^j(z)=\mathrm{Eva}^j_z(\sigma_k)=\langle\!\langle \sigma_k(\cdot),K_j(\cdot,z)\rangle\!\rangle_h.
\]
Using the expansion
\[
K_j(z,w)=\sum_{k=1}^\infty \langle\!\langle K_j(\cdot,w),\sigma_k(\cdot)\rangle\!\rangle_h\,\sigma_k(z)=\sum_{k=1}^\infty \overline{\sigma_k^j(w)}\,\sigma_k(z),
\]
and taking \(w=z\), we obtain
\[
\sum_{k=1}^\infty |\sigma_k^j(z)|^2=K_j^j(z,z)=\|K_j(\cdot,z)\|_h^2.
\]
On the other hand,
\[
\|K_j(\cdot,z)\|_h=\|\mathrm{Eva}^j_z\|=\sup\{|f^j(z)|:f\in H^2(X,E,h),\ \|f\|_h\le1\}
\]
is bounded on any compact subset of \(X\). Therefore \(\sum_{k=1}^\infty|\sigma_k(z)|_h^2\) is uniformly convergent on any compact subset of \(X\).

Now consider, on \(X\times X\), the subsheaves
\[
\mathcal E_N\subset(\mathcal O_{X\times X})^{r^2},
\qquad
\mathcal E_N:=\mathrm{Gen}\left\{\bigl(\sigma_k^i(z)\overline{\sigma_k^j(\bar w)}\bigr)_{1\le i,j\le r}:1\le k\le N\right\}.
\]
The family \(\{\mathcal E_N\}_{N\ge1}\) is an increasing sequence of coherent subsheaves of the coherent sheaf \((\mathcal O_{X\times X})^{r^2}\). By the strong Noetherian property of coherent analytic sheaves, this sequence is locally stationary. Hence, on \(K\times K\), there exists \(N_K\ge1\) such that
\[
\mathcal E_{\infty}:=\mathrm{Gen}\left\{\bigl(\sigma_k^i(z)\overline{\sigma_k^j(\bar w)}\bigr)_{1\le i,j\le r}:k\ge1\right\}
=\mathcal E_{N_K}.
\]
By the Cauchy--Schwarz inequality,
\[
\Bigl|\sum_{k=m_1}^{m_2}\sigma_k^i(z)\overline{\sigma_k^j(\bar w)}\Bigr|^2
\le
\sum_{k=m_1}^{m_2}|\sigma_k^i(z)|^2
\sum_{k=m_1}^{m_2}|\sigma_k^j(\bar w)|^2
\]
for all \(m_2\ge m_1\), so the matrix-valued series
\[
S(z,w):=\sum_{k=1}^\infty \bigl(\sigma_k^i(z)\overline{\sigma_k^j(\bar w)}\bigr)_{1\le i,j\le r}
\]
converges uniformly on \(K\times K\) and hence a holomorphic section of $\mathcal E_{\infty}$. Thus there exist holomorphic functions \(a_k\in\mathcal O(K\times K)\), \(1\le k\le N_K\), such that
\[
S(z,w)
=
\sum_{k=1}^{N_K} a_k(z,w)\,
\bigl(\sigma_k^i(z)\overline{\sigma_k^j(\bar w)}\bigr)_{1\le i,j\le r}
\qquad\text{on }K\times K.
\]
Taking \(w=\bar z\) and contracting with the metric \(h\), we obtain, on \(K\),
\[
\sum_{k=1}^\infty|\sigma_k(z)|_h^2
=
\sum_{k=1}^{N_K} a_k(z,\bar z)\,|\sigma_k(z)|_h^2
\le
C_K\sum_{k=1}^{N_K}|\sigma_k(z)|_h^2,
\]
where \(C_K:=\sup_{z\in K}\sum_{k=1}^{N_K}|a_k(z,\bar z)|\) (shrinking \(K\) if necessary). This completes the proof.
\end{proof}
 
\begin{prop}\label{prop:min-integral}
Let \(D\subset\CC^n\) be a bounded domain, \(E\to D\) a trivial holomorphic vector bundle of rank \(r\), and \(h\) a singular Hermitian metric on \(E\) such that \(|u|_{h^*}\) is upper semi-continuous for every local holomorphic section \(u\) of \(E^*\). Assume that \((E,h)\) satisfies the multiple coarse\(L^2\) extension property.  
Then for  any \(s\in H^0(D,E)\), any \(m\ge2\) and any relatively compact subset \(K\Subset D\),,
\[
\int_K\left(|s|_h^2-\min\{|s|_{h}^2,|s|_{h_m}^{2m/(m-1)}\}\right)dV<+\infty,
\]
where \(h_m^*\) is the approximating Finsler metric on \(E^*\) given by
\[
|u|^2_{h_m^*}:=
\left(
\sum_{j\ge1}|\langle u^{\otimes m},\sigma_j^{(m)}(x)\rangle|^2
\right)^{1/m},
\qquad u\in E_x^*,
\]
with \(\{\sigma_j^{(m)}\}_{j\ge1}\) an orthonormal Hilbert basis of \(H^2(D,E^{\otimes m},h^{\otimes m})\).
\end{prop}

\begin{proof}
Define the measurable set
\[
A_m:=\left\{x\in K:\ |s|_h^2\ge \bigl(|s|_{h_m}^2\bigr)^{\frac{m}{m-1}}\right\}.
\]
It suffices to prove
\[
\int_{A_m}|s|_h^2\,dV<+\infty.
\]
We first derive a pointwise inequality on \(A_m\). Fix \(x\in A_m\) with \(0<|s|_h(x)<+\infty\), there exists \(\zeta_x\in E_x^*\setminus\{0\}\) such that
\[
|\zeta_x|_{h^*}|s|_h(x)=|\langle \zeta_x,s\rangle|=1.
\]
In addition, by definition, we have
\(|s|_{h_m}^2\ge |\zeta_x|_{h_m^*}^{-2}\).
Then 
\(|\zeta_x|_{h^*}^{2m-2}\le |\zeta_x|_{h_m^*}^{2m}\).
Notice that
\[
|\zeta_x|_{h_m^*}^{2m}
=
\sum_{j=1}^\infty |\langle \zeta_x^{\otimes m},\sigma_j^{(m)}(x)\rangle|^2\le\sum_{j=1}^\infty|\zeta_x|_{h^*}^{2m}\,|\sigma_j^{(m)}(x)|_{h^{\otimes m}}^2,
\]
 then we obtain
\[
|s|_h^2(x)=|\zeta_x|_{h^*}^{-2}
\le
\sum_{j=1}^\infty |\sigma_j^{(m)}(x)|_{h^{\otimes m}}^2.
\]

Now apply Lemma~\ref{lem:finite-control} to the bundle \(E^{\otimes m}\) with the Hermitian metric \(h^{\otimes m}\). For every relatively compact subset \(K\Subset D\), there exist \(N_K\ge1\) and \(C_K>0\) such that
\[
\sum_{j=1}^\infty |\sigma_j^{(m)}(x)|_{h^{\otimes m}}^2
\le
C_K\sum_{j=1}^{N_K} |\sigma_j^{(m)}(x)|_{h^{\otimes m}}^2,
\qquad x\in K.
\]
Since \(0<|s|_h(x)<+\infty\) for almost every \(x\in D\),
 we obtain
\[
\int_{A_m}|s|_h^2\,dV
\le
C_K\sum_{j=1}^{N_K}\int_K |\sigma_j^{(m)}|_{h^{\otimes m}}^2\,dV
\le
C_K\sum_{j=1}^{N_K}\|\sigma_j^{(m)}\|_{h^{\otimes m}}^2<+\infty.
\]
We complete the proof.
\end{proof}

\begin{prop}[=Proposition \ref{prop:multiplier-sheaves-strong-openness}]
Let \(D\subset\CC^n\) be a bounded domain, \(E\to D\) a trivial holomorphic vector bundle of rank \(r\), and \(h\) a Hermitian metric on \(E\) such that \(|u|_{h^*}\) is upper semi-continuous for every local holomorphic section \(u\) of \(E^*\). Assume that \((E,h)\) satisfies the multiple coarse \(L^2\) extension property and the strong openness property: for every relatively compact open subset \(U\Subset D\) over which \(E\) is trivial, there exists \(\eps>0\) such that
\[
\mathcal E(h)_x=\mathcal E(h(\det h)^\eps)_x
\]
for all \(x\in U\). Then \(
\mathcal E(h)=\mathcal E\!\left(h_m^{m/(m-1)}\right)
\) for $m$ large enough.
\end{prop}

\begin{proof}
By \cite[Lemma 3.8]{LXYZ24}, the strong openness property is equivalent to the following: if \(|s|_h^2\) is integrable near a point \(x\), then there exists \(p>1\) such that \(|s|_h^{2p}\) is integrable near \(x\). Since \(\int_U|s|_h^2\,dV<+\infty\), there exists \(p>1\) such that
\[
\int_U |s|_h^{2p}\,dV<+\infty.
\]
Choose \(m\ge2\) such that \(m/(m-1)<p\), i.e., \(2m/(m-1)<2p\). By Theorem~\ref{thm:abstract}, we have the upper bound
$|s|_{h_m}\le C_r^{1/(2m)}|s|_h$ 
locally on relatively compact subsets.  Therefore
\[
\int_U |s|_{h_m}^{2m/(m-1)}\,dV<+\infty.
\]

Conversely, assume that for some \(m\ge2\),
\[
\int_U |s|_{h_m}^{2m/(m-1)}\,dV<+\infty.
\]
 Notice that
\begin{align*}
   & \int_U \left||s|_h^2-|s|_{h_m}^{2m/(m-1)}\right|\,dV   \\
  \le & \int_U \left(|s|_h^2-\min\{|s|_h^2,|s|_{h_m}^{2m/(m-1)}\}\right)dV
+
\int_U |s|_{h_m}^{2m/(m-1)}\,dV.
\end{align*}
Then by Proposition~\ref{prop:min-integral}, we get  that 
 \(\int_D |s|_h^2\,dV<+\infty\).  We complete the proof.
 \end{proof}

\begin{remark}\label{rem:multiplier-sheaves-reverse}
It should be noted that the object \(h_m^{m/(m-1)}\) is not a Finsler metric. Indeed, a Finsler metric is required to be homogeneous of degree one, whereas \(h_m^{m/(m-1)}\) is homogeneous of degree \(m/(m-1)\), and \(m/(m-1)\neq 1\) for \(m\ge2\). To obtain a genuine Finsler metric, fix a smooth Hermitian metric \(g\) on \(E\) and define
\[
\tilde h_m:=h_m^{m/(m-1)}\,g^{-1/(m-1)}.
\]
Then \(\tilde h_m\) is a Finsler metric, and  
\(
\mathcal E(\tilde h_m)=\mathcal E\!\left(h_m^{m/(m-1)}\right).
\)
\end{remark}
 
\subsection{Proof of Theorem \ref{thm:second-main-hermitian}}
Let  \(D\subset\CC^n\)  be a bounded pseudoconvex domain, $E=D\times\CC^r$ be a  trivial holomorphic vector bundle and \(h\)  a Griffiths semi-positive singular Hermitian metric on  \(E\). Fix a global trivialization of \(E\), so that \(\det E\) is trivialized and \(\det h\) is a positive function on \(D\). 
Fix a local holomorphic section \(u\) of \(E^*\), a point \(x\in D\) with $\det h(x)<+\infty$, and \(r>0\) such that \(B(x,r)\Subset D\). Notice that $$|u|_{h_m^*}=\left( \sup_{\|s\|_{L^2(g_m)}\le1}|\langle u^{\otimes m},s(x)\rangle|  \right)^{1/m},$$ where \(s\in H^2(D,F_m,g_m)\) and \((F_m, g_m)=(\mathrm{Sym}^mE\otimes\det E, \mathrm{Sym}^mh\otimes\det h)\). By the mean value inequality,
\[
|\langle u^{\otimes m},s (x)\rangle|
\le
C_r\int_{B(x,r)}|\langle u^{\otimes m},s(y)\rangle| \,dV(y),
\]
where \(C_r=1/\operatorname{Vol}(B(x,r))\). By the duality inequality,
\[
|\langle u^{\otimes m},s \rangle| 
\le
|u^{\otimes m}|_{\mathrm{Sym}^mh\otimes \det h^*} |s |_{g_m}
=
|u|_{h^*}^m |s |_{g_m}(\det h)^{-1}.
\]
Therefore,
 $$
|\langle u^{\otimes m},s(x)\rangle| 
\le
C_r\sup_{B(x,r)}|u|_{h^*}^m (\det h)^{-1}$$
  and hence we obtain the upper bound.

For the lower bound, by the definition of the dual norm, there exists \(a\in E_x\) such that
\[
\langle u,a\rangle=1,\qquad |a|_{h(x)}=|u|_{h^*(x)}^{-1}.
\]
Since $h$ is Griffiths semi-positive,  by the \(L^2\) extension property from \cite[Theorem 1.3 and Proposition 6.2]{LXYZ24}, there exists \(s\in H^2(D,F_m,g_m)\) such that
\[
s(x)=a^{\otimes m}, \qquad
\|s\|_{g_m}^2
\le
C |a|_{h(x)}^{2m}\det h(x).
\]
Then \(\|s\|_{g_m}\le C^{1/2}|u|_{h^*}^{-m}\det h(x)^{1/2}\). Consider \(s'=s/\|s\|_{g_m}\). Then \(\|s'\|_{g_m}=1\), and
\[
|\langle u^m,s'(x)\rangle| 
\ge C^{-1/2}|u|_{h^*}^{m}\det h(x)^{-1/2}.
\]
Therefore, we obtain the lower bound.  We complete the proof.

\vskip1em

\textbf{Acknowledgement:}
The author  was supported by the National Natural	Science Foundation of China (No.12501101) and the Scientific Research Foundation of Chongqing University of Technology (No.2025ZDZ013).

\end{document}